\documentclass[11pt]{amsart}
\usepackage[utf8]{inputenc}
\usepackage[T1]{fontenc}
\usepackage{amsfonts}
\usepackage{amssymb}
\usepackage{color}
\usepackage{graphicx,psfrag,rotating}
\usepackage{amsmath}
\usepackage{mathptmx}
\usepackage{relsize}

\newcommand\Z{\mathbb{Z}}

\newcommand\I{{\mathcal L}}

\allowdisplaybreaks[1]
\theoremstyle{plain}

\newtheorem{theorem}{Theorem}[section]

\newtheorem{lemma}[theorem]{Lemma}

\newtheorem*{main thm}{Main Theorem}

\theoremstyle{definition}
\newtheorem{definition}[theorem]{Definition}
\newtheorem{example}[theorem]{Example}

\newtheorem{remark}[theorem]{Remark}

\theoremstyle{remark}

\begin{document}

\newenvironment{prooff}{\medskip \par \noindent {\it Proof}\ }{\hfill
$\square$ \medskip \par}
    \def\sqr#1#2{{\vcenter{\hrule height.#2pt
        \hbox{\vrule width.#2pt height#1pt \kern#1pt
            \vrule width.#2pt}\hrule height.#2pt}}}
    \def\square{\mathchoice\sqr67\sqr67\sqr{2.1}6\sqr{1.5}6}
\def\pf#1{\medskip \par \noindent {\it #1.}\ }
\def\endpf{\hfill $\square$ \medskip \par}
\def\demo#1{\medskip \par \noindent {\it #1.}\ }
\def\enddemo{\medskip \par}
\def\qed{~\hfill$\square$}


\title[Multiple Curves on Punctured Orientable Surfaces]
{Multiple Curves on Punctured Orientable Surfaces}

\author[A. Meral]
{Alev Meral}

\subjclass[2010]{Primary:57N16, 57N37; Secondary:57N05}
\keywords{ Multiple Curves, Geometric Intersection Number, Punctured Orientable Surfaces }

\address{(A. M.) Department of Mathematics, D\.{ı}cle University, 21280
D\.{ı}yarbak{\i}r, Turkey}
\email{$alev.meral@dicle.edu.tr$}

\date{\today}

\begin{abstract}
We describe each multiple curve on the orientable surface of genus-$g$ with $n$ punctures and one boundary component by using this multiple curve's geometric intersection number with the embedded curves in this surface. 
\end{abstract}

\maketitle

\setcounter{secnumdepth}{2}
\setcounter{section}{0}

\section{Introduction}
One of the naive ways to describe multiple curves, which are the disjoint unions of finitely many essential simple closed curves on the standard punctured disk modulo isotopy,  is to use the geometric intersection numbers between the multiple curve and embedded arcs in the disk \cite{dynnikov02}. In \cite{meral19}, this way is generalized for such curve sistems on the orientable surface of genus-$1$ with $n$ ~($n \geq 2$) punctures and one boundary component. The coordinate system \cite{dynnikov02} obtained using this method was extensively used to solve various dynamical and combinatorial problems such as the word problem in the braid group \cite{dehornoy02}, \cite{dehornoy08} and calculate the topological entropy of an braid \cite{moussafir06}. The aim of this paper is to generalize the way which describes each multiple curve by using the geometric intersection numbers with the  embedded curves in the punctured orientable genus-$1$ surface with one boundary to the orientable surface of genus-$g$ ~($g \geq 1$) with $n$ punctures and one boundary component.

Throughout the paper, $S_{n,g}$ shall denote a genus-$g$ ~($g \geq 1$) surface with $n ~(n \geq 1)$ punctures and one boundary component. In order to describe a given multiple curve on $S_{n,g}$, a system consisting of $3n+7g-5$ arcs and $g$ simple  closed curves on $S_{n,g}$ is used. Given a multiple curve $\I$, we shall introduce a vector in $\Z^{3n+8g-5}_{\geq 0} \setminus \{0\}$ by using the geometric intersection numbers with the curves in our system and consider the linear combinations of these intersection numbers (see Section~\ref{intersection_numbers}).

\section{Geometric Intersection Numbers with Customized Curves Embedded in $S_{n,g}$}\label{intersection_numbers}
In this section, we shall describe the multiple curves on  $S_{n,g}$, whose geometric intersection numbers with the customized curves embedded in $S_{n,g}$ and directions are given. For this, we use the model shown in Figure~\ref{gen_model}.
\begin{figure}[!ht]
\centering
\psfrag{a1}[tl]{\scalebox{0.7}{$\scriptstyle{\alpha_{1}}$}}
\psfrag{a2}[tl]{\scalebox{0.7}{$\scriptstyle{\alpha_{2}}$}} 
\psfrag{a(2i-3)}[tl]{\scalebox{0.7}{$\scriptstyle{\alpha_{2i-3}}$}} 
\psfrag{a(2i-2)}[tl]{\scalebox{0.7}{$\scriptstyle{\alpha_{2i-2}}$}} 
\psfrag{a(2i-1)}[tl]{\scalebox{0.7}{$\scriptstyle{\alpha_{2i-1}}$}}
\psfrag{a(2i)}[tl]{\scalebox{0.7}{$\scriptstyle{\alpha_{2i}}$}} 
\psfrag{a(2i+1)}[tl]{\scalebox{0.7}{$\scriptstyle{\alpha_{2i+1}}$}} 
\psfrag{a(2i+2)}[tl]{\scalebox{0.7}{$\scriptstyle{\alpha_{2i+2}}$}} 
\psfrag{a(2n-1)}[tl]{\scalebox{0.7}{$\scriptstyle{\alpha_{2n-1}}$}}
\psfrag{a(2n)}[tl]{\scalebox{0.7}{$\scriptstyle{\alpha_{2n}}$}} 
\psfrag{b1}[tl]{\scalebox{0.7}{$\scriptstyle{\beta_{1}}$}} 
\psfrag{bi}[tl]{\scalebox{0.7}{$\scriptstyle{\beta_{i}}$}} 
\psfrag{b(i+1)}[tl]{\scalebox{0.7}{$\scriptstyle{\beta_{i+1}}$}} 
\psfrag{b(n+1)}[tl]{\scalebox{0.7}{$\scriptstyle{\beta_{n+1}}$}} 
\psfrag{b(n+2)}[tl]{\scalebox{0.7}{$\scriptstyle{\beta_{n+2}}$}} 
\psfrag{b(n+3)}[tl]{\scalebox{0.7}{$\scriptstyle{\beta_{n+3}}$}} 
\psfrag{b(n+g)}[tl]{\scalebox{0.7}{$\scriptstyle{\beta_{n+g}}$}} 
\psfrag{b(n+g-1)}[tl]{\scalebox{0.7}{$\scriptstyle{\beta_{n+g-1}}$}} 
\psfrag{b'(n+2)}[tl]{\scalebox{0.7}{$\scriptstyle{\beta^{'}_{n+2}}$}} 
\psfrag{b'(n+3)}[tl]{\scalebox{0.7}{$\scriptstyle{\beta^{'}_{n+3}}$}} 
\psfrag{b'(n+g)}[tl]{\scalebox{0.7}{$\scriptstyle{\beta^{'}_{n+g}}$}} 
\psfrag{b'(n+g-1)}[tl]{\scalebox{0.7}{$\scriptstyle{\beta^{'}_{n+g-1}}$}} 
\psfrag{c1}[tl]{\scalebox{0.7}{$\scriptstyle{c_1}$}}
\psfrag{c2}[tl]{\scalebox{0.7}{$\scriptstyle{c_2}$}}
\psfrag{c(g-1)}[tl]{\scalebox{0.7}{$\scriptstyle{c_{g-1}}$}}
\psfrag{cg}[tl]{\scalebox{0.7}{$\scriptstyle{c^{*}}$}}
\psfrag{g1}[tl]{\scalebox{0.7}{$\scriptstyle{\gamma_1}$}}
\psfrag{g2}[tl]{\scalebox{0.7}{$\scriptstyle{\gamma_2}$}}
\psfrag{gg}[tl]{\scalebox{0.7}{$\scriptstyle{\gamma_g}$}}
\psfrag{g(g-1)}[tl]{\scalebox{0.7}{$\scriptstyle{\gamma_{g-1}}$}}
\psfrag{k1}[tl]{\scalebox{0.7}{$\scriptstyle{\xi_1}$}}
\psfrag{k2}[tl]{\scalebox{0.7}{$\scriptstyle{\xi_2}$}}
\psfrag{k3}[tl]{\scalebox{0.7}{$\scriptstyle{\xi_3}$}}
\psfrag{k4}[tl]{\scalebox{0.7}{$\scriptstyle{\xi_4}$}}
\psfrag{k1'}[tl]{\scalebox{0.7}{$\scriptstyle{\xi^{'}_1}$}}
\psfrag{k2'}[tl]{\scalebox{0.7}{$\scriptstyle{\xi^{'}_2}$}}
\psfrag{k3'}[tl]{\scalebox{0.7}{$\scriptstyle{\xi^{'}_3}$}}
\psfrag{k4'}[tl]{\scalebox{0.7}{$\scriptstyle{\xi^{'}_4}$}}
\psfrag{k(2g-3)}[tl]{\scalebox{0.7}{$\scriptstyle{\xi_{2g-3}}$}}
\psfrag{k(2g-2)}[tl]{\scalebox{0.7}{$\scriptstyle{\xi_{2g-2}}$}}
\psfrag{k'(2g-3)}[tl]{\scalebox{0.7}{$\scriptstyle{\xi^{'}_{2g-3}}$}}
\psfrag{k'(2g-2)}[tl]{\scalebox{0.7}{$\scriptstyle{\xi^{'}_{2g-2}}$}}
\psfrag{d(2i-1)}[tl]{\scalebox{0.7}{$\scriptstyle{{\color{blue}\Delta_{2i-1}}}$}}
\psfrag{d(2i)}[tl]{\scalebox{0.7}{$\scriptstyle{{\color{blue}\Delta_{2i}}}$}}
\psfrag{d(2i-2)}[tl]{\scalebox{0.7}{$\scriptstyle{{\color{blue}\Delta_{2i-2}}}$}}
\psfrag{d(2i+1)}[tl]{\scalebox{0.7}{$\scriptstyle{{\color{blue}\Delta_{2i+1}}}$}}
\psfrag{d(2n)}[tl]{\scalebox{0.7}{$\scriptstyle{{\color{blue}\Delta_{2n}}}$}}
\psfrag{d1}[tl]{\scalebox{0.7}{$\scriptstyle{{\color{blue}\Delta_{1}}}$}}
\includegraphics[width=1.15\textwidth]{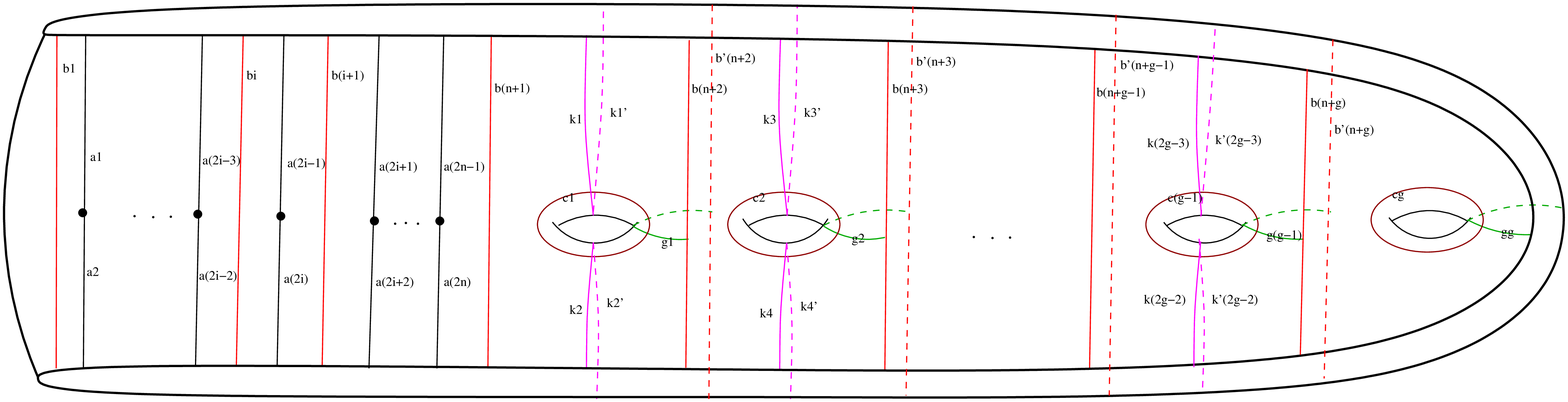}
\caption{ Curves on $S_{n,g}$ }\label{gen_model}
\end{figure} 
Here, the endpoints of arcs $\alpha_i  ~(1 \leq i \leq 2n)$, $\beta_i  ~(1 \leq i \leq n+g)$, $\beta^{'}_i  ~(n+2 \leq i \leq n+g)$, $\xi_i  ~(1 \leq i \leq 2g-2)$ and $\xi^{'}_i  ~(1 \leq i \leq 2g-2)$ are either on the boundary or on the puncture. While $c_i  ~(1 \leq i \leq g-1)$ and $c^{*}$ are the longitude of each torus respectively, $\gamma_i  ~(1 \leq i \leq g)$ is the arc whose both endpoints are on  $\beta_i$ and $\beta^{'}_i$ as depicted in Figure~\ref{gen_model} and  $\gamma_g$ is the arc whose both endpoints are on the boundary. Also, note that each $\gamma_i ~(1 \leq i \leq g-1)$ and $\gamma_g$ intersects each $c_i$ and $c^{*}$ respectively once transversally.

Let  $\I_{n,g}$ be the set of multiple curves on $S_{n,g}$ and $\I \in \I_{n,g}$.  Throughout the paper, we always work with the minimal representative (a multiple curve in the same isotopy class intersecting the customized curves embedded in $S_{n,g}$ minimally) of $\I$ and denote it by $L$. Let the vector $(\alpha_{1}, \cdots, \alpha_{2n}; \beta_{1}, \cdots, \beta_{n+g}; \beta^{'}_{n+2}, \cdots, \beta^{'}_{n+g}; \xi_{1}, \cdots, \linebreak \xi_{2g-2}; \xi^{'}_{1}, \cdots, \xi^{'}_{2g-2}; \gamma_{1}, \cdots, \gamma_{g}; c_{1}, \cdots, c_{g-1}; c^{*}) \in \{\Z^{3n+8g-5}_{\geq 0}\} \setminus \{0\}$ show the intersection numbers of $L$ with the corresponding arcs and the simple closed curves $c_{i}$ and $c^{*}.$ For example, $(5, 2, 5, 2, 4, 3; 7, 5, 7, 1, 5, 5; 5, 3; 6, 3, 5, 2; 4, 1, 4, 1; 2, 2, 3; 2, 0; 3)$ are the intersection numbers of the multiple curve $L$ depicted in Figure~\ref{inter_num_w_coor_curv}.
\begin{figure}[!ht]
\centering
\includegraphics[width=0.83\textwidth]{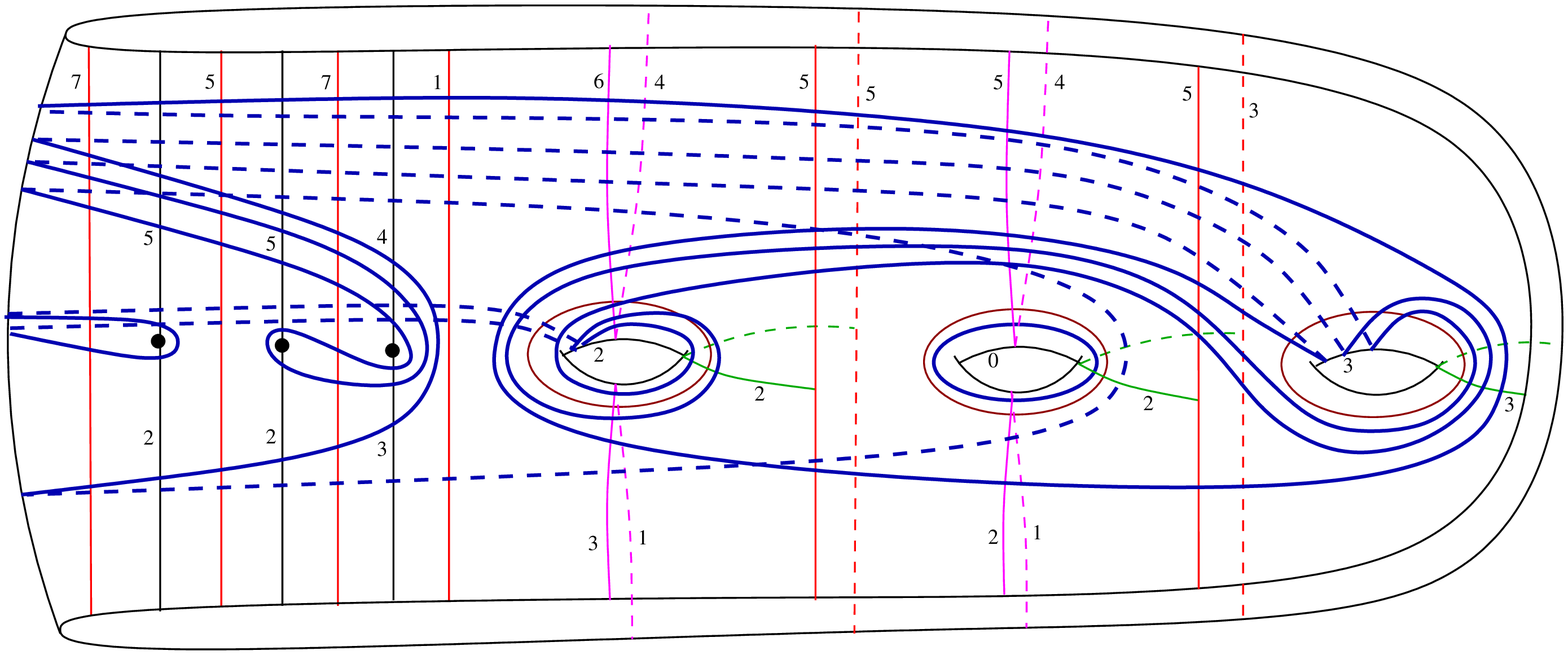}
\caption{ Intersection numbers with the curves embedded in $S_{3,3}$ }\label{inter_num_w_coor_curv}
\end{figure} 

\subsection{Path Components on $S_{n,g}$}
In this section, we shall introduce the path components of a multiple curve $L$ on $S_{n,g}$ and derive formulas for the number of these components.

Let $U_i  ~(1 \leq i \leq n)$ be the region that is bounded by $\beta_i$ and $\beta_{i+1}$, $G_{i}  ~(1 \leq i \leq g-1)$ be the region bounded by $\beta_{n+i}$, $\beta^{'}_{n+i}$, $\beta_{n+i+1}$ and $\beta^{'}_{n+i+1}$, and $G^{*}$ be the region bounded by $\beta_{n+g}$, $\beta^{*}_{n+g}$ and the boundary of $S_{n,g}$ ~($\partial S_{n,g}$). Each component of $L\cap U_i$, $L\cap G_i$  and $L\cap G^{*}$ is called the \emph{path component} of $L$ in  $U_i$, $G_i$ and $G^{*}$, respectively. Since $L$ is minimal, there are $4$ types of path components in the region $U_i$  as on the disk \cite{dynnikov02} (see Figure~\ref{puncture_bilesen_gen}). An \emph{above component} has endpoints on $\beta_i$ and $\beta_{i+1}$ and intersects $\alpha_{2i-1}$. A  \emph{below component} has endpoints on $\beta_i$ and $\beta_{i+1}$ and intersects $\alpha_{2i}$. A \emph{left loop component} has both endpoints  on $\beta_{i+1}$ and intersects $\alpha_{2i-1}$ and $\alpha_{2i}$ (Figure~\ref{puncture_bilesen_gen}a). A \emph{right loop component} has both endpoints on $\beta_i$ and intersects $\alpha_{2i-1}$ and $\alpha_{2i}$ (Figure~\ref{puncture_bilesen_gen}b).  There are $6$ types of path components in the region $G^{*}$. The first three of these are \emph{curve $c^{*}$}, which is the longitude of the torus in $G^{*}$ (Figure~\ref{genus_loop_gen}a); \emph{visible genus component}, which has both endpoints  on $\beta_{n+1}$  and does not intersect the curve $c^{*}$ (Figure~\ref{genus_loop_gen}b); \emph{invisible genus component}, which has both endpoints on $\beta_{1}$ and does not intersect the curve $c^{*}$  (Figure~\ref{genus_loop_gen}c).  The other three components are called \emph{twist}, which have endpoints on $\beta_{n+g}$ and $\beta^{'}_{n+g}$ and  intersect the curve $c^{*}$ (see Figure~\ref{cikis_yon_gen}). These components are non-twist, negative twist  and positive twist  components. The \emph{non-twist  component} does not make any twist  (see Figure~\ref{cikis_yon_gen}a). The \emph{negative twist component}  makes clockwise twist (see Figure~\ref{cikis_yon_gen}b). The \emph{positive twist  component}  makes counterclockwise twist (see Figure~\ref{cikis_yon_gen}c) \cite{meral19}. There are $14$ types of path components in each region $G_{i}$. These are \emph{curve $c_{i}$}, which is the longitude of the torus in $G_{i}$ (similar to Figure~\ref{genus_loop_gen}a); \emph{visible-left genus component}, which has both endpoints  on $\beta_{n+i+1}$  and does not intersect the curve $c_{i}$ (Figure~\ref{diagonal_left_right_loop}a); \emph{invisible-left genus component}, which has both endpoints  on $\beta^{'}_{n+i+1}$  and does not intersect the curve $c_{i}$ (Figure~\ref{diagonal_left_right_loop}a); \emph{visible-right genus component}, which has both endpoints  on $\beta_{n+i}$  and does not intersect the curve $c_{i}$ (Figure~\ref{diagonal_left_right_loop}b); \emph{invisible-right genus component}, which has both endpoints  on $\beta^{'}_{n+i}$  and does not intersect the curve $c_{i}$ (Figure~\ref{diagonal_left_right_loop}b); \emph{upper diagonal component}, which has endpoints on $\beta^{'}_{n+i}$ and $\beta_{n+i+1}$ and intersects the  curve $c_{i}$ and the arc $\xi_{2i-1}$ (see Figure~\ref{diagonal_left_right_loop}c); \emph{lower diagonal component}, which has endpoints on $\beta^{'}_{n+i}$ and $\beta_{n+i+1}$ and intersects the  curve $c_{i}$ and the arc $\xi_{2i}$ (see Figure~\ref{diagonal_left_right_loop}d); \emph{visible above component}, which has endpoints on $\beta_{n+i}$ and $\beta_{n+i+1}$ and intersects the arc $\xi_{2i-1}$ (see Figure~\ref{diagonal_left_right_loop}e); \emph{invisible above component}, which has endpoints on $\beta^{'}_{n+i}$ and $\beta^{'}_{n+i+1}$ and intersects the arc $\xi^{'}_{2i-1}$ (see Figure~\ref{diagonal_left_right_loop}e); \emph{visible below component}, which has endpoints on $\beta_{n+i}$ and $\beta_{n+i+1}$ and intersects the arc $\xi_{2i}$ (see Figure~\ref{diagonal_left_right_loop}f); \emph{invisible below component}, which has endpoints on $\beta^{'}_{n+i}$ and $\beta^{'}_{n+i+1}$ and  intersects the arc $\xi^{'}_{2i}$ (see Figure~\ref{diagonal_left_right_loop}f); \emph{negative twist component}, which has endpoints on $\beta^{'}_{n+i}$ and $\beta_{n+i}$ or $\beta^{'}_{n+i}$ and $\beta_{n+i+1}$ and intersects the curve $c_{i}$ and makes clockwise twist (see Figures~\ref{twist_components_gen}a and \ref{twist_components_gen}b); \emph{positive twist component}, which has endpoints on $\beta^{'}_{n+i}$ and $\beta_{n+i}$ or $\beta^{'}_{n+i}$ and $\beta_{n+i+1}$ and intersects the curve $c_{i}$ and makes counterclockwise twist (see Figures~\ref{twist_components_gen}c and \ref{twist_components_gen}d);  and \emph{non-twist  component} (see Figure~\ref{twist_components_gen}e).
\begin{figure}[!ht]
\centering
\psfrag{a}[tl]{\scalebox{0.6}{$\scriptstyle{a}$}}
\psfrag{b}[tl]{\scalebox{0.6}{$\scriptstyle{b}$}} 
\psfrag{a2i}[tl]{\scalebox{0.6}{$\scriptstyle{\alpha_{2i}}$}} 
\psfrag{a2i-1}[tl]{\scalebox{0.6}{$\scriptstyle{\alpha_{2i-1}}$}}
\psfrag{bi}[tl]{\scalebox{0.6}{$\scriptstyle{\beta_{i}}$}} 
\psfrag{bi+1}[tl]{\scalebox{0.6}{$\scriptstyle{\beta_{i+1}}$}} 
\psfrag{U(i)}[tl]{\scalebox{0.6}{$\scriptstyle{U_{i}}$}}
\includegraphics[width=0.6\textwidth]{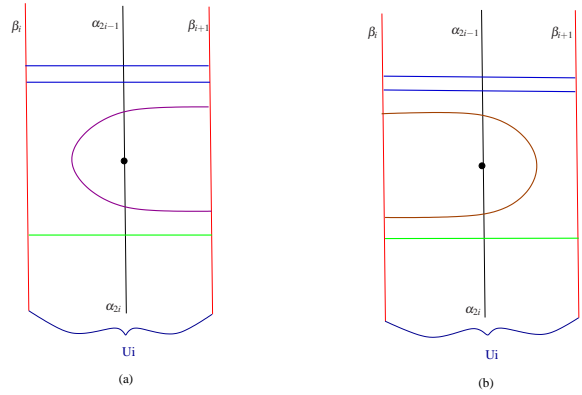}
\caption{Above and below components, left and right loop components in the region $U_i$}\label{puncture_bilesen_gen}
\end{figure}

\begin{figure}[!ht]
\centering
\psfrag{a}[tl]{\scalebox{0.55}{$\scriptstyle{a}$}}
\psfrag{b}[tl]{\scalebox{0.55}{$\scriptstyle{b}$}} 
\psfrag{g}[tl]{\scalebox{0.6}{$\scriptstyle{\gamma_{g}}$}}
\psfrag{t}[tl]{\scalebox{0.6}{$\scriptstyle{c^{*}}$}} 
\psfrag{bn+1}[tl]{\scalebox{0.6}{$\scriptstyle{\beta_{n+g}}$}}
\psfrag{b1}[tl]{\scalebox{0.6}{$\scriptstyle{\beta^{'}_{n+g}}$}} 
\includegraphics[width=0.97\textwidth]{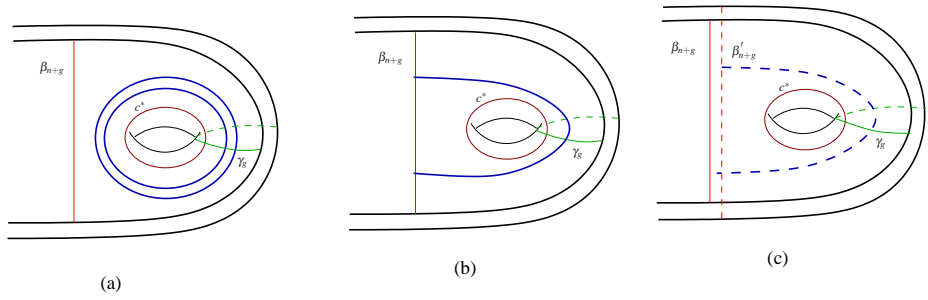}
\caption{(a) $c^{*}$ curves, (b) Visible genus component, (c) Invisible genus component in the region $G^{*}$ }\label{genus_loop_gen}
\end{figure}

\begin{figure}[!ht]
\centering
\psfrag{a}[tl]{\scalebox{0.55}{$\scriptstyle{a}$}}
\psfrag{b}[tl]{\scalebox{0.55}{$\scriptstyle{b}$}} 
\psfrag{g}[tl]{\scalebox{0.6}{$\scriptstyle{\gamma_{g}}$}} 
\psfrag{t}[tl]{\scalebox{0.6}{$\scriptstyle{c^{*}}$}} 
\psfrag{bn+1}[tl]{\scalebox{0.6}{$\scriptstyle{\beta_{n+g}}$}}
\psfrag{b1}[tl]{\scalebox{0.6}{$\scriptstyle{\beta^{'}_{n+g}}$}} 
\includegraphics[width=0.97\textwidth]{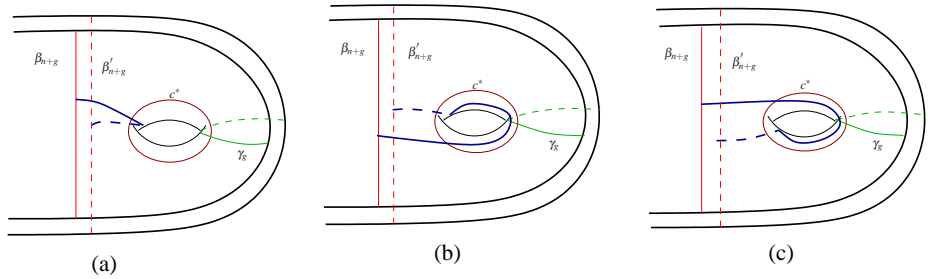}
\caption{ (a) Non-twist component, (b) Negative twist  component, (c) Positive twist  component.}\label{cikis_yon_gen}
\end{figure}

\begin{figure}[!ht]
\centering
\psfrag{a}[tl]{\scalebox{0.55}{$\scriptstyle{a}$}}
\psfrag{b}[tl]{\scalebox{0.55}{$\scriptstyle{b}$}} 
\psfrag{g}[tl]{\scalebox{0.6}{$\scriptstyle{\gamma_{i}}$}} 
\psfrag{x(2i-1)}[tl]{\scalebox{0.6}{$\scriptstyle{\xi_{2i-1}}$}} 
\psfrag{x'(2i-1)}[tl]{\scalebox{0.6}{$\scriptstyle{\xi^{'}_{2i-1}}$}} 
\psfrag{x(2i)}[tl]{\scalebox{0.6}{$\scriptstyle{\xi_{2i}}$}} 
\psfrag{x'(2i)}[tl]{\scalebox{0.6}{$\scriptstyle{\xi^{'}_{2i}}$}} 
\psfrag{t}[tl]{\scalebox{0.6}{$\scriptstyle{c_{i}}$}} 
\psfrag{b(n+i)}[tl]{\scalebox{0.6}{$\scriptstyle{\beta_{n+i}}$}}
\psfrag{b'(n+i)}[tl]{\scalebox{0.6}{$\scriptstyle{\beta^{'}_{n+i}}$}} 
\psfrag{b(n+i+1)}[tl]{\scalebox{0.6}{$\scriptstyle{\beta_{n+i+1}}$}}
\psfrag{b'(n+i+1)}[tl]{\scalebox{0.6}{$\scriptstyle{\beta^{'}_{n+i+1}}$}} 
\includegraphics[width=0.75\textwidth]{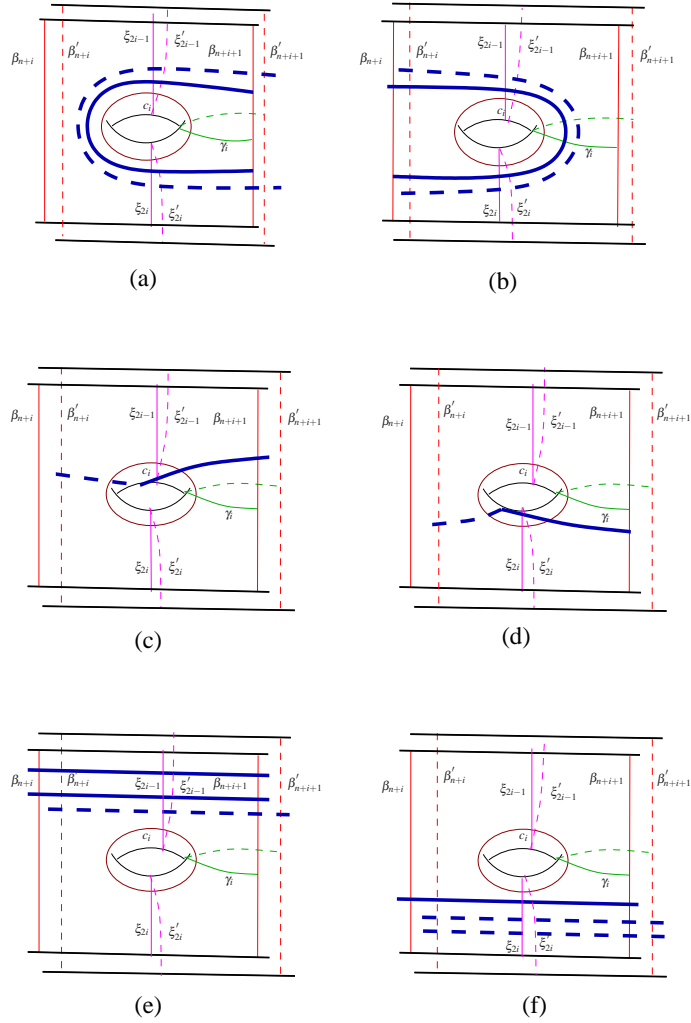}
\caption{ (a)  Visible-left and invisible-left genus components, (b) Visible-right and invisible-right genus components, (c) Upper diagonal component, (d) Lower diagonal component, (e) Visible above and invisible above components, (f) Visible below and invisible below components in the region $G_{i}$}\label{diagonal_left_right_loop}
\end{figure}

\begin{figure}[!ht]
\centering
\psfrag{a}[tl]{\scalebox{0.55}{$\scriptstyle{a}$}}
\psfrag{b}[tl]{\scalebox{0.55}{$\scriptstyle{b}$}} 
\psfrag{g}[tl]{\scalebox{0.6}{$\scriptstyle{\gamma_{i}}$}} 
\psfrag{x(2i-1)}[tl]{\scalebox{0.6}{$\scriptstyle{\xi_{2i-1}}$}} 
\psfrag{x'(2i-1)}[tl]{\scalebox{0.6}{$\scriptstyle{\xi^{'}_{2i-1}}$}} 
\psfrag{x(2i)}[tl]{\scalebox{0.6}{$\scriptstyle{\xi_{2i}}$}} 
\psfrag{x'(2i)}[tl]{\scalebox{0.6}{$\scriptstyle{\xi^{'}_{2i}}$}} 
\psfrag{t}[tl]{\scalebox{0.6}{$\scriptstyle{c_{i}}$}} 
\psfrag{b(n+i)}[tl]{\scalebox{0.6}{$\scriptstyle{\beta_{n+i}}$}}
\psfrag{b'(n+i)}[tl]{\scalebox{0.6}{$\scriptstyle{\beta^{'}_{n+i}}$}} 
\psfrag{b(n+i+1)}[tl]{\scalebox{0.6}{$\scriptstyle{\beta_{n+i+1}}$}}
\psfrag{b'(n+i+1)}[tl]{\scalebox{0.6}{$\scriptstyle{\beta^{'}_{n+i+1}}$}} 
\includegraphics[width=0.75\textwidth]{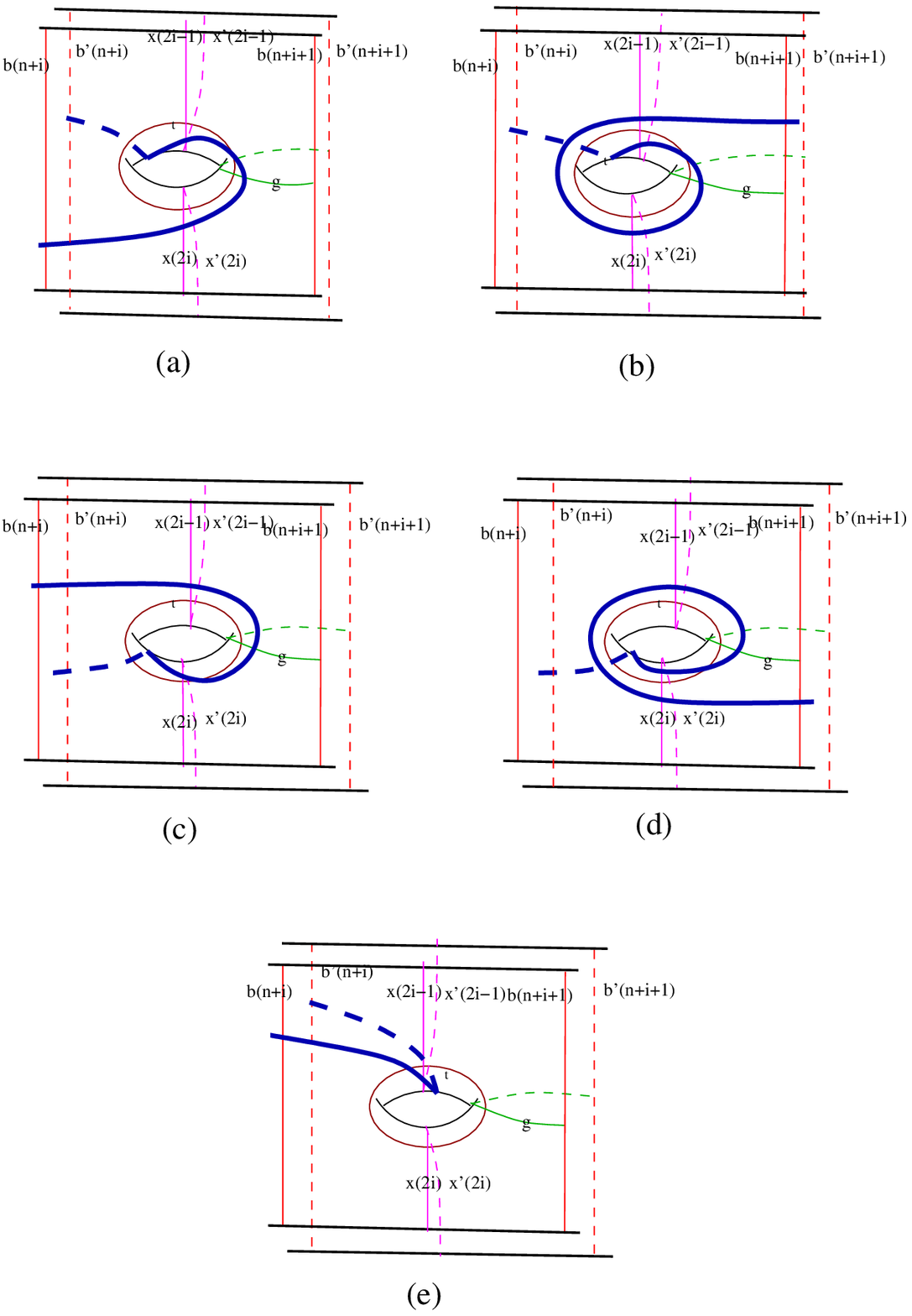}
\caption{ (a) and (b) Negative twist  component; (c) and (d) Positive twist  component; (e) Non-twist component in $G_{i}$.}\label{twist_components_gen}
\end{figure}

\begin{remark}\label{assuming}
For ease of calculation, throughout the paper, we shall assume that each diagonal  component (Figures~\ref{diagonal_left_right_loop}c and \ref{diagonal_left_right_loop}d) and twist component (Figure~\ref{twist_components_gen}) on $S_{n,g}$ intersect the arc $\xi_{2i-1}$ instead of the arc $\xi^{'}_{2i-1}$ and the arc $\xi_{2i}$ instead of the arc $\xi^{'}_{2i}$. Also, we shall assume that the invisible (dashed) parts of these components are only on the invisible-left side of $S_{n,g}$, as seen in the corresponding figures and that each $G_{i}$ has only one of the upper diagonal component or the lower diagonal component.

\end{remark}

\begin{remark}\label{not_c_cutting_gen}
Since a multiple curve  $L \in \I_{n,g}$ consists of the simple closed curves that do not intersect each other, there cannot be both  curve  $c_{i}$  and twist or diagonal components at the same time in the region $G_{i}$, and both curve $c^{*}$ and  twist components at the same time in the region $G^{*}$.
\end{remark}

\begin{definition}
Let $d^{u}_{2i-1}$ and $d^{l}_{2i}$ give the number of the upper and lower diagonal components in the region $G_{i}$ for $1 \leq i \leq g-1$, respectively. Also, let $c^{'}_{i}$ denote the number of the twist components in $G_{i}$. Thus,  throughout the paper, $c_{i}$ shall be defined as the sum of these components. That is, 
\begin{equation}\label{number_cutting}
c_{i} = c^{'}_{i} + d^{u}_{2i-1} + d^{l}_{2i}.
\end{equation}
Note that since there cannot be any diagonal components in $G^{*}$, here  $c_{i}$ shall be equal to only the number of the twist components in $G^{*}$, and in this case we shall denote $c_{i}$ with the number $c^{*}$.

\end{definition}

\begin{definition}\label{def_twist_num}
A twist component's \emph{twist number} is the signed number of intersections with the arc $\gamma_{i}  ~(1 \leq i \leq g)$.
\end{definition}

\begin{remark}\label{not_greater_1}
Since a multiple curve on $S_{n,g}$ does not contain any self-intersections, the  directions of the twists  have to be the same. Also, in the regions $G_{i}$ and $G^{*}$, the difference between the twist numbers of two different twist components  cannot be greater than $1$ \cite{meral19}.

If we denote the \emph{smaller twist number} by $t_{i}$ and the \emph{bigger twist number} by $t_{i}+1$, then the \emph{total twist number} ~$T_{i} ~(1 \leq i \leq g)$  in 
$G_{i} ~(1 \leq i \leq g-1)$ and $G^{*}$ is the sum of the twist numbers of twist components (see Figure~\ref{twist_components_gen}). Hence,  if the difference between the twist numbers of any two twist components is $0$, then  
$$T_{i} = t_{i}(c_{i} - d^{u}_{2i-1} - d^{l}_{2i}).$$  
On the other hand, if the difference between the twist numbers of any   two twist components is $1$, then $$ T_{i} = m_{i}(t_{i} + 1) + (c_{i} - d^{u}_{2i-1} - d^{l}_{2i} - m_{i})t_{i}, $$  \noindent 
where $m_{i} \in \Z_{\geq0}$ is the number of the twist components with the twist number  $t_{i} + 1$, and  $ c_{i} - d^{u}_{2i-1} - d^{l}_{2i} - m_{i} $ is the number of  the twist components with the twist number  $t_{i}$.
\end{remark}

\begin{remark}
Although $T_{i}$ gives the total twist number in each region $G_{i}$ and $G^{*}$, it cannot show the directions of twists by itself. Therefore, we first calculate the number of each $T_{i}$, and then we add a sign in front of $T_{i}$, denoting the negative direction by $-T_{i}$ and the positive direction by $T_{i}$. However, since only the total number of twists is required in the formulas throughout the paper, $|T_{i}|$ shall be used as the total number of twists in order not to cause any confusion.
\end{remark}

Now, we calculate the path components  of $L$ in the regions $G_{i}$ for $1 \leq i \leq g-1$ and  $G^{*}$ for $i = g$.

\begin{lemma}\label{lem_genus_loop_gen}
Let $L$ be given with the intersection numbers  $(\alpha; \beta; \beta^{'}; \xi; \xi^{'}; \gamma; c; c^{*})$, and the number of visible genus components and the number of invisible genus components in $G_{i}$ be $l_{i}$ and $l_{i}^{'} $, respectively. Also, let the number of visible genus components and the number of invisible genus components in $G^{*}$ be $l_{g}$ and $l_{g}^{'} $, respectively. Then for  $1 \leq i \leq g-1$,
\begin{eqnarray*}
l_{i} = \max\{0, \frac{|\beta_{n+i} - \beta_{n+i+1}| - c_{i}}{2}\} \quad \text{ and } \quad l_{g} = \frac{\beta_{n+g} - c^{*}}{2},
\end{eqnarray*}
\noindent and for $2 \leq i \leq g-1$,
\begin{eqnarray*}
l_{1}^{'} = \max\{0, \frac{|\beta_{1} - \beta_{n+2}^{'}| - c_{1}}{2}\}, \quad
l_{i}^{'} = \max\{0, \frac{|\beta^{'}_{n+i} - \beta^{'}_{n+i+1}| - c_{i}}{2}\} 
\end{eqnarray*}
\noindent and
\begin{eqnarray*}
l_{g}^{'} = \frac{\beta^{'}_{n+g} - c^{*}}{2}.
\end{eqnarray*}
Note that if $\beta_{n+i} < \beta_{n+i+1}$, the visible genus component in $G_{i}$ is left; if $\beta_{n+i} > \beta_{n+i+1}$, the visible genus component in $G_{i}$ is right. Similarly, if $\beta^{'}_{n+i} < \beta^{'}_{n+i+1}$, the invisible genus component in $G_{i}$ is left; if $\beta^{'}_{n+i} > \beta^{'}_{n+i+1}$, the invisible genus component in $G_{i}$ is right.
\end{lemma}

\begin{proof}
The absolute value of the difference between the intersection numbers on the arcs $\beta_{n+i}$ and $\beta_{n+i+1}$, namely $|\beta_{n+i} - \beta_{n+i+1}|$, gives us the sum of twist, diagonal and visible genus component numbers. If $\beta_{n+i} < \beta_{n+i+1}$, the arc $\beta_{n+i+1}$ intersects once with each twist component (Figures~\ref{twist_components_gen}b and  \ref{twist_components_gen}d) or diagonal component (Figures~\ref{diagonal_left_right_loop}c and \ref{diagonal_left_right_loop}d) and twice with each visible-left genus component (Figure~\ref{diagonal_left_right_loop}a). Let  the number of visible-left genus components  and the number of visible-right genus components be denoted by $l^{L}_{i}$  and $l^{R}_{i}$, respectively.  Hence,  $$\beta_{n+i+1} - \beta_{n+i} = c^{'}_{i} + d^{u}_{2i-1} + d^{l}_{2i} + 2l^{L}_{i}.$$  \noindent From Equation~(\ref{number_cutting}), $$ \beta_{n+i+1} - \beta_{n+i} = c_{i} + 2l^{L}_{i}.$$
\noindent Since a multiple curve consists of the simple closed curves that do not intersect  each other, this curve system contains only one of the visible-left genus components or visible-right genus components. Therefore, we can denote the number of both visible-left genus components and visible-right genus components as $l_{i}$. Thus, we can write 
\begin{equation}\label{visible-left}
\beta_{n+i+1} - \beta_{n+i} = c_{i} + 2l_{i}.
\end{equation}

If $\beta_{n+i} > \beta_{n+i+1}$, the arc $\beta_{n+i}$ intersects once each twist component (Figures~\ref{twist_components_gen}a, \ref{twist_components_gen}c and  \ref{twist_components_gen}e) and twice each visible-right genus component (Figure~\ref{diagonal_left_right_loop}b). From Remark~\ref{assuming}, there cannot be any diagonal component; otherwise, self-intersections occur in this curve system. Since $c_{i} = c^{'}_{i} + d^{u}_{2i-1} + d^{l}_{2i}$, here $$ \beta_{n+i} - \beta_{n+i+1} = c_{i} + 2l^{R}_{i}.$$  \noindent That is, 
\begin{equation}\label{visible-right}
\beta_{n+i} - \beta_{n+i+1} = c_{i} + 2l_{i}.
\end{equation}
\noindent From Equations~(\ref{visible-left}) and (\ref{visible-right}), we can write $ |\beta_{n+i} - \beta_{n+i+1}| = c_{i} + 2l_{i}$. Therefore, $l_{i} = \frac{|\beta_{n+i} - \beta_{n+i+1}| - c_{i}}{2}$. When $c_{i} \geq |\beta_{n+i} - \beta_{n+i+1}|$, there cannot be any visible genus component in multiple curve. Hence $l_{i} =  \max\{0, \frac{|\beta_{n+i} - \beta_{n+i+1}| - c_{i}}{2}\}$ is derived. Similarly, we can find $l^{'}_{1}$ and $l^{'}_{i}$. For the proofs of $l_{g}$ and $l^{'}_{g}$, you can look at \cite{meral19}.

\end{proof}

In the following lemma, we calculate the total twist number of twist components in each $G_{i}$ and $G^{*}$:

\begin{lemma}\label{lem_total_twist_gen}
Let $L$ be given with the intersection numbers  $(\alpha; \beta; \beta^{'}; \xi; \xi^{'}; \gamma; c; c^{*})$, denoting the signed total twist number of twist components in each $G_{i}$ and $G^{*}$ by $T_{i}$ and $T_{g}$, respectively. For ~$2 \leq i \leq g-1$, we have 
\begin{align}\label{total_twist_Ti}
|T_{i}|&= \left\{ \begin{array}{ll}
       0& \mbox{if $c_{i} = 0$},\\
         \gamma_{i} - \max\{0, \frac{\max\{0, \beta_{n+i} - \beta_{n+i+1}\} - c_{i}}{2}\} - \max\{0, \frac{\max\{0, \beta^{'}_{n+i} - \beta^{'}_{n+i+1}\} - c_{i}}{2}\} & \mbox{if $c_{i} \neq 0$.}   
         \end{array} \right.
\end{align}

\noindent For $i = 1$,
\begin{align}\label{total_twist_T1}
|T_{1}|&= \left\{ \begin{array}{ll}
       0& \mbox{if $c_{1} = 0$},\\
         \gamma_{1} - \max\{0, \frac{\max\{0, \beta_{n+1} - \beta_{n+2}\} - c_{1}}{2}\} - \max\{0, \frac{\max\{0, \beta_{1} - \beta^{'}_{n+2}\} - c_{1}}{2}\} & \mbox{if $c_{1} \neq 0$.}   
         \end{array} \right.
\end{align}

\noindent For $i = g$,
\begin{align}\label{total_twist_Tg}
|T_{g}|&= \left\{ \begin{array}{ll}
       0& \mbox{if $c^{*} = 0$},\\
         \gamma_{g} - \frac{\beta_{n+g} - c^{*}}{2} - \frac{\beta^{'}_{n+g} - c^{*}}{2} & \mbox{if $c^{*} \neq 0$.}   
         \end{array} \right.
\end{align}

\noindent
The sign of the negative twist component is $-1$ and the sign of the positive twist component is $1$. 
\end{lemma}

\begin{proof}
Let us denote the total twist number of the twist components of $L$ in each $G_{i}$ by $|T_{i}|$.  Observe that the curve $\gamma_{i}$ intersects once the curve $c_{i}$ (Figure~\ref{genus_loop_gen}a) 
and it intersects once each visible-right  and invisible-right genus components (Figure~\ref{diagonal_left_right_loop}b).  Also, $\gamma_{i}$ intersects  $L$ by the total number of twists of the twist components (Figure~\ref{twist_components_gen}). However, from Remark~\ref{not_c_cutting_gen}, there cannot be 
twists and  the curve $c_{i}$ in $G_{i}$ at the same time.  Therefore, when $c_{i} \neq 0$, we have 
\begin{equation}\label{gama_kesenler}
\gamma_{i} = l_{i} + l_{i}^{'} + |T_{i}|,
\end{equation}
\noindent where $l_{i}$, ~$l_{i}^{'}$ and $|T_{i}|$ denote the number of visible-right genus, invisible-right genus components and the total twist number of the twist components in $G_{i}$, respectively.

Since a multiple curve consists of the simple closed curves that do not intersect  each other and from Definition~\ref{def_twist_num}, we can write
\begin{equation*}
\gamma_{i} = \max\{0, \frac{\max\{0, \beta_{n+i} - \beta_{n+i+1}\} - c_{i}}{2}\} + \max\{0, \frac{\max\{0, \beta^{'}_{n+i} - \beta^{'}_{n+i+1}\} - c_{i}}{2}\} + |T_{i}|.
\end{equation*}
\noindent Hence, we have Equality~(\ref{total_twist_Ti}) as follows
\begin{equation*}
|T_{i}| = \gamma_{i} - \max\{0, \frac{\max\{0, \beta_{n+i} - \beta_{n+i+1}\} - c_{i}}{2}\} - \max\{0, \frac{\max\{0, \beta^{'}_{n+i} - \beta^{'}_{n+i+1}\} - c_{i}}{2}\}.
\end{equation*}

Equalities~(\ref{total_twist_T1}) and (\ref{total_twist_Tg}) can be obtained similar to the above calculations and \cite{meral19}, respectively.

\end{proof}

\begin{remark}\label{when_diagonal_exist}
When there is one of the upper diagonal components or 
lower diagonal components in the region $G_{i}$, 
the  equation $c_{i} = d_{2i-1}^{u} + d_{2i}^{l} + |T_{i}|$ is 
used so that the curves on the surface do not intersect. In this case, $|T_{i}|$ cannot be greater than $c_i$.	
	
\end{remark}

By using the following lemma, we calculate the number of  the curves $c_{i}$ and $c^{*}$ (Figure~\ref{genus_loop_gen}a) in each region $G_{i} ~(1 \leq i \leq g-1)$ and $G^{*}$, respectively.
\begin{lemma}\label{c_curves_gen}
Let $L$ be given with the intersection numbers  $(\alpha; \beta; \beta^{'}; \xi; \xi^{'}; \gamma; c; c^{*})$. We find the number of  the curves $c_{i}$ and $c^{*}$ in $L$, denoting by $p(c_{i})$ and $p(c^{*})$, as follows. For $2 \leq i \leq g-1$, 
\begin{align}\label{number_pci}
p(c_{i})&= \left\{ \begin{array}{ll}
       \gamma_{i} - \max\{0, \frac{\max\{0, \beta_{n+i} - \beta_{n+i+1}\}}{2}\} - \max\{0, \frac{\max\{0, \beta^{'}_{n+i} - \beta^{'}_{n+i+1}\}}{2}\}& \mbox{if $c_{i} = 0$},\\
         0& \mbox{if $c_{i} \neq 0$.}   
         \end{array} \right.
\end{align}

\noindent For $i = 1$,
\begin{align}\label{number_pc1}
p(c_{1})&= \left\{ \begin{array}{ll}
       \gamma_{1} - \max\{0, \frac{\max\{0, \beta_{n+1} - \beta_{n+2}\}}{2}\} - \max\{0, \frac{\max\{0, \beta_{1} - \beta^{'}_{n+2}\}}{2}\}& \mbox{if $c_{1} = 0$},\\
         0& \mbox{if $c_{1} \neq 0$.}   
         \end{array} \right.
\end{align}

\noindent For $i = g$,
\begin{align}\label{number_pcg}
p(c^{*})&= \left\{ \begin{array}{ll}
       \gamma_{g} - \frac{\beta_{n+g}}{2} - \frac{\beta^{'}_{n+g}}{2} & \mbox{if $c^{*} = 0$},\\
         0& \mbox{if $c^{*} \neq 0$.}   
         \end{array} \right.
\end{align}

\end{lemma}

\begin{proof}
Whenever $c_{i} = 0$, we have $\gamma_{i} = l_{i} + l_{i}^{'} + p(c_{i})$. Since a multiple curve consists of the simple closed curves that do not intersect each other and from Definition~\ref{def_twist_num}, we can write
\begin{eqnarray*}
	\gamma_{i} = \max\{0, \frac{\max\{0, \beta_{n+i} - \beta_{n+i+1}\}}{2}\} + \max\{0, \frac{\max\{0, \beta^{'}_{n+i} - \beta^{'}_{n+i+1}\}}{2}\} + p(c_{i}).
\end{eqnarray*}
Hence, $p(c_{i}) = \gamma_{i} - \max\{0, \frac{\max\{0, \beta_{n+i} - \beta_{n+i+1}\}}{2}\} - \max\{0, \frac{\max\{0, \beta^{'}_{n+i} - \beta^{'}_{n+i+1}\}}{2}\} $ is derived.

Equalities~(\ref{number_pc1}) and (\ref{number_pcg}) can be obtained similar to the above calculations and \cite{meral19}, respectively.
\end{proof}

In the following lemma, we find the number of the upper diagonal components, $d_{2i-1}^{u}$, and the lower diagonal components, $d_{2i}^{l}$, in each region $G_{i} ~(1 \leq i \leq g-1)$.
\begin{lemma}\label{lem_diagonals}
Let $L$ be given with the intersection numbers  $(\alpha; \beta; \beta^{'}; \xi; \xi^{'}; \gamma; c; c^{*})$, and the number of the upper diagonal components and the number of the lower diagonal components in $G_{i}$ be $d_{2i-1}^{u}$ and $d_{2i}^{l}$, respectively.  Then for  $1 \leq i \leq g-1$,	
\begin{equation}\label{upper_diagonal}
d_{2i-1}^{u} = \max\{c_{i} - |T_{i}|, T_{i}c_{i}\} - \max\{0, T_{i}c_{i}\}  
\end{equation}
\noindent and	
\begin{equation}\label{lower_diagonal}
d_{2i}^{l} = \max\{c_{i} - |T_{i}|, -T_{i}c_{i}\} - \max\{0, -T_{i}c_{i}\}.  
\end{equation}	
	
\end{lemma}

\begin{proof}
	Firstly, we assume that there are upper diagonal components in the region $G_{i}$.  When  $T_{i} < 0$, from Remark~\ref{when_diagonal_exist}, we see $ d_{2i-1}^{u} + d_{2i}^{l} = c_{i} - |T_{i}|$. From Remark~\ref{assuming}, $G_{i}$ has no lower diagonal components. Therefore, we can write $ d_{2i-1}^{u} = c_{i} - |T_{i}|$. When $T_{i} > 0$, it should be $ d_{2i-1}^{u} = 0$ so that the curves do not intersect  each other. The equation~(\ref{upper_diagonal}) provides these properties completely.
	
	When there are lower diagonal components in $G_{i}$, we can find the equation~(\ref{lower_diagonal}) similar to the number of upper diagonal components.
\end{proof}

The twist numbers of each twist component of a multiple curve whose intersection numbers are given are found by using Remark~\ref{not_greater_1} and Lemma~\ref{lem_total_twist_gen}, which we find these twist numbers with the following lemma. The proof of this lemma is similar to the proof in \cite{meral19}.
\begin{lemma}\label{each_twist}
	Let $L$ be given with the intersection numbers  
	$(\alpha; \beta; \beta^{'}; \xi; \xi^{'}; \gamma; c; c^{*})$. 
	Let $|T_{i}| ~(1 \leq i \leq g-1)$ and $|T_{g}|$ be the total twist numbers
	 in each regions $G_{i}$ and $G^{*}$, respectively.
	  Also, let $m_{i}$ and $m^{*}$ be the number of twist components, each with 
	   $t_{i}+1$ and $t_{g}+1$ twists and 
	  $c_{i} - d^{u}_{2i-1} - d^{l}_{2i} - m_{i}$ and $c^{*} - m^{*}$ be 
	  the number of twist components, each with  $t_{i}$ and $t_{g}$ twists
	   in each $G_{i}$ and $G^{*}$, respectively. 
	   In this case,
	
	\begin{equation}
	m_{i} \equiv |T_{i}| ~(mod ~(c_{i} - d^{u}_{2i-1} - d^{l}_{2i}))  \quad \text{ and } \quad t_{i} = \frac{|T_{i}| - m_{i}}{c_{i} - d^{u}_{2i-1} - d^{l}_{2i}}
	\end{equation}
	\noindent and
	\begin{equation}
	m^{*} \equiv |T_{g}| ~(mod ~c^{*})  \quad \text{ and } \quad t_{g} = \frac{|T_{g}| - m^{*}}{c^{*}},
	\end{equation}
	\noindent where $c_{i} - d^{u}_{2i-1} - d^{l}_{2i} \neq 0$ and $c^{*} \neq 0$. 
\end{lemma}

In Lemma~\ref{looking_right_left}, we shall define some auxiliary components that shall be used to calculate the number of the visible above components denoted by $u_{2i-1}^{va}$ and the number of the visible below components denoted by $u_{2i}^{vb}$ in the rest of the paper.

\begin{lemma}\label{looking_right_left}
Let $L$ be given with the intersection numbers  $(\alpha; \beta; \beta^{'}; \xi; \xi^{'}; \gamma; c; c^{*})$. Then for  $1 \leq i \leq g-1$,	if $\beta_{n+i} \leq \beta_{n+i+1}$, the number of the intersections of twist components together with total diagonals with the arc $\beta_{n+i+1}$, denoting by $n_{i}$, is as follows.
\begin{equation}
n_{i} = \frac{\beta_{n+i+1} - \beta_{n+i} + c_{i}}{2} - \max\{0, \frac{|\beta_{n+i} - \beta_{n+i+1}| - c_{i}}{2}\}.
\end{equation}

\noindent Hence, we can find the number of the intersections of twist components together with total diagonals with the arc $\beta_{n+i}$ as $c_{i} - n_{i}$.

On the other hand, if $\beta_{n+i} \geq \beta_{n+i+1}$, the number of the intersections of twist components together with total diagonals with the arc $\beta_{n+i}$, denoting by $k_{i}$, is as  follows.
\begin{equation}
k_{i} = \frac{\beta_{n+i} - \beta_{n+i+1} + c_{i}}{2} - \max\{0, \frac{|\beta_{n+i} - \beta_{n+i+1}| - c_{i}}{2}\}.
\end{equation}

 \noindent Hence, we can find the number of the intersections of twist components together with total diagonals with the arc $\beta_{n+i+1}$ as $c_{i} - k_{i}$.

\end{lemma}

\begin{proof}
When $\beta_{n+i} \leq \beta_{n+i+1}$, we can write the number of the intersections on the arcs  $\beta_{n+i+1}$ and $\beta_{n+i}$ as follows:
\begin{equation}\label{right1}
\beta_{n+i+1} = n_{i} + 2l_{i} + u_{2i-1}^{va} + u_{2i}^{vb},
\end{equation}
\begin{equation}\label{left1}
\beta_{n+i} = c_{i} - n_{i} + u_{2i-1}^{va} + u_{2i}^{vb}.
\end{equation}
From equations ~(\ref{right1}) and ~(\ref{left1}), we derive 
\begin{equation*}
n_{i} = \frac{\beta_{n+i+1} - \beta_{n+i} + c_{i}}{2} - l_{i}.
\end{equation*}

\noindent When $\beta_{n+i} \geq \beta_{n+i+1}$, we can find $k_{i}$ similar to $n_{i}$.

\end{proof}

\begin{remark}\label{remark_loop}
	In each region $U_i$,  for $1 \leq i \leq n$, let the number of the loop components be denoted by $|b_i|$, where 
	\begin{equation}
	b_{i} = \frac{\beta_{i} - \beta_{i+1}}{2}.
	\end{equation}
	If $b_i < 0$, the loop component is called \emph{left}; if $b_i > 0$, the loop component is called \emph{right} \cite{dynnikov02}.
\end{remark}

Now, we find the number of above and below components in each $U_i ~(1 \leq i \leq n)$ and the number of visible above, visible below, invisible above and invisible below components in each $G_i ~(1 \leq i \leq g-1)$.
\begin{lemma}\label{all_above_below}
Let $L$ be given with the intersection numbers  $(\alpha; \beta; \beta^{'}; \xi; \xi^{'}; \gamma; c; c^{*})$. Also, let the number of above and below components in each $U_i$ and the number of visible above, visible below, invisible above and invisible below components in each $G_i$ be denoted by $u_{2i-1}^{a}, ~u_{2i}^{b}, ~u_{2i-1}^{va}, ~u_{2i}^{vb}, ~u_{2i-1}^{v'a}$ and $u_{2i}^{v'b}$, respectively. Then for $1 \leq i \leq n$,
\begin{equation}\label{above_below_puncture}
u_{2i-1}^{a} = \alpha_{2i-1} - |b_i| \quad \mbox{ and } \quad u_{2i}^{b} =  \alpha_{2i} - |b_i|.
\end{equation}

 For $1 \leq i \leq g-1$, if $|T_i| \neq 0$,

\noindent when $\beta_{n+i} \leq \beta_{n+i+1}$,
\begin{equation}\label{denk_1_ab_vis_dif_zero}
u_{2i-1}^{va} = \xi_{2i-1} - |T_i| - \max\{n_i - d_{2i}, T_i\} + \max\{0, T_i\} - l_i,
\end{equation}
\begin{equation}\label{denk_1_bel_vis_dif_zero}
u_{2i}^{vb} = \xi_{2i} - |T_i| - \max\{n_i - d_{2i-1}, -T_i\} + \max\{0, -T_i\} - l_i;
\end{equation}

\noindent when $\beta_{n+i} \geq \beta_{n+i+1}$,
\begin{equation}\label{denk_2_ab_vis_dif_zero}
u_{2i-1}^{va} = \xi_{2i-1} - |T_i| - \max\{c_i - k_i - d_{2i}, T_i\} + \max\{0, T_i\} - l_i,
\end{equation}
\begin{equation}\label{denk_2_bel_vis_dif_zero}
u_{2i}^{vb} = \xi_{2i} - |T_i| - \max\{c_i - k_i - d_{2i-1}, -T_i\} + \max\{0, -T_i\} - l_i.
\end{equation}

If $|T_i| = 0$,
\begin{equation}\label{denk_ab_vis_zero}
u_{2i-1}^{va} = \xi_{2i-1} - \max\{p(c_i), d_{2i-1}\} - l_i,
\end{equation}
\begin{equation}\label{denk_bel_vis_zero}
u_{2i}^{vb} = \xi_{2i} - \max\{p(c_i), d_{2i}\} - l_i.
\end{equation}

Also,
\begin{equation}
u_{2i-1}^{v'a} = \xi^{'}_{2i-1} - l^{'}_i \quad \mbox{ and } \quad u_{2i}^{v'b} = \xi^{'}_{2i} - l^{'}_i.
\end{equation}
	
\end{lemma}

\begin{proof}
	The proofs of Equations~(\ref{above_below_puncture}) are obvious since each above and below component intersects $\alpha_{2i-1}$ and $\alpha_{2i}$, respectively (see Figure~\ref{puncture_bilesen_gen}).
	
	Let $|T_i| \neq 0$. When $\beta_{n+i} \leq \beta_{n+i+1}$, from Lemma~\ref{looking_right_left}, the number of the intersections of twist components together with total diagonal components with the arc $\beta_{n+i+1}$ is $n_i$. When we subtract the number of lower diagonal components (Figure~\ref{diagonal_left_right_loop}d) from $n_i$, the arc $\xi_{2i-1}$ intersects $n_i - d_{2i}$ times with the twist components. The arc $\xi_{2i-1}$ also intersects $l_i$ times with the visible genus components (Figures~\ref{diagonal_left_right_loop}a and \ref{diagonal_left_right_loop}b), $u_{2i-1}^{va}$ times with the visible above components (Figure~\ref{diagonal_left_right_loop}e), and by the total number of twists, $|T_i|$. When $T_i > 0$, $\xi_{2i-1}$ intersects by the total number of twists; whereas when $T_i < 0$, $\xi_{2i-1}$ intersects by the total number of twists and $n_i - d_{2i}$. That is,
	$$ \xi_{2i-1} = |T_i| + \max\{n_i - d_{2i}, T_i\} - \max\{0, T_i\} + l_i + u_{2i-1}^{va}.$$  \noindent Hence, we get Equation~(\ref{denk_1_ab_vis_dif_zero}) as follows. $$ u_{2i-1}^{va} = \xi_{2i-1} - |T_i| - 
	\max\{n_i - d_{2i}, T_i\} +  \max\{0, T_i\} - l_i. $$
	
	Similarly, in addition to the number of visible genus components and visible below components, when $T_i > 0$, $\xi_{2i}$ intersects by the total number of twists and $n_i - d_{2i-1}$; whereas when $T_i < 0$, $\xi_{2i}$ intersects by the total number of twists. Hence, $$ \xi_{2i} =  |T_i| + 
	\max\{n_i - d_{2i-1}, -T_i\} - \max\{0, -T_i\} + l_i + u_{2i}^{vb}. $$ \noindent From here, we can write Equation~(\ref{denk_1_bel_vis_dif_zero}) as follows.
	 $$ u_{2i}^{vb} = \xi_{2i} - |T_i| - \max\{n_i - d_{2i-1}, -T_i\} + 
	 \max\{0, -T_i\} - l_i. $$
	 
	 When $\beta_{n+i} \geq \beta_{n+i+1}$, we can derive the  Equations~(\ref{denk_2_ab_vis_dif_zero}) and (\ref{denk_2_bel_vis_dif_zero}) similar to the Equations~(\ref{denk_1_ab_vis_dif_zero}) and (\ref{denk_1_bel_vis_dif_zero}) using Lemma~\ref{looking_right_left}.
	 
	 Let $|T_i| = 0$. In this case, in addition to  visible genus components and visible above components, $\xi_{2i-1}$ intersects either the curve $c_i$ or the upper diagonal components (see Remark~\ref{not_c_cutting_gen}). That is, 
	 $$ \xi_{2i-1} = u_{2i-1}^{va} + \max\{p(c_i), d_{2i-1}\} + l_i. $$ \noindent Thus, we find Equation~(\ref{denk_ab_vis_zero}) as $u_{2i-1}^{va} = \xi_{2i-1} - \max\{p(c_i), d_{2i-1}\} - l_i.$ Similarly, $u_{2i}^{vb}$ is derived.
	 
	 From Remark~\ref{assuming}, $\xi^{'}_{2i-1}$ intersects only invisible genus components and invisible above components, and $\xi^{'}_{2i}$ intersects only invisible genus components and invisible below components. Thus, we can write 
	 $$ u^{v'a}_{2i-1} = \xi^{'}_{2i-1} - l'_i \quad \mbox{ and } \quad  u^{v'b}_{2i} = \xi^{'}_{2i} - l'_i.$$ 
	\end{proof}

\begin{example}\label{exp}
Let $(6, 2, 4, 2, 5, 1; 8, 6, 4, 6, 7, 2; 3, 0; 5, 4, 6, 6; 4, 1, 0, 0; 2, 5, 3; 3, 3; 0)$ be the intersection numbers of a multiple curve $L \in \I_{3,3}$ with the corresponding arcs and the simple closed curves $c_i$ and $c^{*}$ in $S_{3,3}$. Also, $T_{1} > 0$ and  $T_{2} < 0$.  We shall show how we draw $L$ from the given intersection numbers. 	

First, we  find the number of each path component in each region $G_i$  for $i = 1, 2$ and $G^{*}$, respectively. From Lemma~\ref{lem_genus_loop_gen}, 
\begin{eqnarray*}
	l_{1} = \max\{0, \frac{|\beta_{4} - \beta_{5}| - c_{1}}{2}\} = \max\{0, \frac{|6 - 7| - 3}{2}\} = 0.
\end{eqnarray*}
	
\noindent Similarly, we have  $l_{2} = 1$, $l_{3} = 1$, $l^{'}_{1} = 1$, $l^{'}_{2} = 0$ and $l^{'}_{3} = 0$. Namely, there is $1$ right-invisible genus component, however there is not any visible genus component in the region $G_{1}$. In $G_{2}$, there is ~$1$ right-visible genus component and no invisible genus component. In $G^{*}$, there is ~$1$ visible genus component and no invisible genus component.  

According to Lemma~\ref{lem_total_twist_gen}, 
\begin{eqnarray*}
|T_{1}| = 2 - \max\{0, \frac{\max\{0, 6 - 7\} - 3}{2}\} - \max\{0, \frac{\max\{0, 8 - 3\} - 3}{2}\} = 1.
\end{eqnarray*}	
\noindent Similarly, $|T_{2}| = 4$ and since $c^{*} = 0$, $|T_{3}| = 0$. That is, the total twist number of the twist components in the region $G_1$ is $1$. The total twist number of the twist components in $G_2$ is 4, however there is not any twist in $G^{*}$. From Lemma~\ref{c_curves_gen}, we observe that since $c_1 \neq 0$ and $c_2 \neq 0$, there are no $c_1$ and $c_2$ curves in the regions $G_1$ and $G_2$. We have $p(c^{*}) = 2$. Therefore, there are $2$ $c^{*}$  curves in $G^{*}$.

We can find the number of upper and lower diagonal components using Lemma~\ref{lem_diagonals} in each $G_i, ~i=1,2$. We know that $T_1 > 0$. Thus, 
\begin{equation*}
d_{1}^{u} = \max\{c_{1} - |T_{1}|, T_{1}c_{1}\} - \max\{0, T_{1}c_{1}\} =  \max\{3 - 1, 1\times3\} - \max\{0, 1\times3\} = 0,
\end{equation*}

\begin{equation*}
d_{2}^{l} = \max\{c_{1} - |T_{1}|, -T_{1}c_{1}\} - \max\{0, -T_{1}c_{1}\} =  \max\{3 - 1, -1\times3\} - \max\{0, -1\times3\} = 2.
\end{equation*}	
\noindent While there are $2$ lower diagonal components in the region $G_1$, there are no upper diagonal components. From Remark~\ref{when_diagonal_exist}, since $|T_2|$ is greater than $c_2$, there are not both diagonal components in $G_2$.

We calculate the twist numbers of each twist component of $L$ in each $G_i$ and $G^{*}$ by Lemma~\ref{each_twist}. In $G_1$,
\begin{equation*}
m_{1} = |T_{1}| ~(mod ~(c_{1} - d^{u}_{1} - d^{l}_{2})) = 1 ~(mod ~(3 - 0 - 2)) = 0,
\end{equation*}

\begin{equation*}
t_{1} = \frac{|T_{1}| - m_{1}}{c_{1} - d^{u}_{1} - d^{l}_{2}} = \frac{1 - 0}{3 - 0 - 2} = 1
\end{equation*}
\noindent and
\begin{equation*}
c_{1} - d^{u}_{1} - d^{l}_{2} - m_{1} = 3 - 0 - 2 - 0 = 1.
\end{equation*}
\noindent Therefore, there is $1$ twist component which has $1$ twist, however there is not any twist component with $t_1 + 1 = 1 + 1 = 2$ twists in $G_1$. In $G_2$,
\begin{equation*}
m_{2} = |T_{2}| ~(mod ~(c_{2} - d^{u}_{3} - d^{l}_{4})) = 4 ~(mod ~(3 - 0 - 0)) = 1,
\end{equation*}

\begin{equation*}
t_{2} = \frac{|T_{2}| - m_{2}}{c_{2} - d^{u}_{3} - d^{l}_{4}} = \frac{4 - 1}{3 - 0 - 0} = 1
\end{equation*}
\noindent and
\begin{equation*}
c_{2} - d^{u}_{3} - d^{l}_{4} - m_{2} = 3 - 0 - 0 - 1 = 2.
\end{equation*}
\noindent Thus, there are $2$ twist components, each with $1$ twist and $1$ twist component which has $t_2 + 1 = 1 + 1 = 2$ twists in $G_2$. Since $c^{*} = 0$, there is no twist in $G^{*}$.

According to Lemma~\ref{looking_right_left}, due to $\beta_4 < \beta_5$, 
\begin{equation*}
n_{1} = \frac{\beta_{5} - \beta_{4} + c_{1}}{2} - \max\{0, \frac{|\beta_{4} - \beta_{5}| - c_{1}}{2}\} = \frac{7 - 6 + 3}{2} - \max\{0, \frac{|6 - 7| - 3}{2}\} = 2.
\end{equation*}
\noindent Hence, the number of the intersections of twist components together with total diagonals with the arc $\beta_5$ in $G_1$ is $2$. The number of the intersections of twist components together with total diagonals with the arc $\beta_4$ in $G_1$ is $c_1 - n_1 = 3 - 2 = 1$.

In $G_2$, due to $\beta_5 > \beta_6$,
\begin{equation*}
k_{2} = \frac{\beta_{5} - \beta_{6} + c_{2}}{2} - \max\{0, \frac{|\beta_{5} - \beta_{6}| - c_{2}}{2}\} = \frac{7 - 2 + 3}{2} - \max\{0, \frac{|7 - 2| - 3}{2}\} = 3.
\end{equation*}
\noindent Thus, the number of the intersections of twist components together with total diagonals with the arc $\beta_5$ in $G_2$ is $3$. The number of the intersections of twist components together with total diagonals with the arc $\beta_6$ in $G_2$ is $c_2 - k_2 = 3 - 3 = 0$.

We find the loop components in each region $U_i, ~i=1,2,3$ by Remark~\ref{remark_loop}.
\begin{equation*}
b_{1} = \frac{\beta_{1} - \beta_{2}}{2} = \frac{8 - 6}{2} = 1,
\end{equation*}
\begin{equation*}
b_{2} = \frac{\beta_{2} - \beta_{3}}{2} = \frac{6 - 4}{2} = 1,
\end{equation*}
\begin{equation*}
b_{3} = \frac{\beta_{3} - \beta_{4}}{2} = \frac{4 - 6}{2} = -1.
\end{equation*}
\noindent Namely, there is $1$ right loop component in $U_1$, $1$ right loop component in $U_2$ and $1$ left loop component in $U_3$.

We calculate the number of above and below components in each $U_i ~(1 \leq i \leq 3)$ and the number of visible above, visible below, invisible above and invisible below components in each $G_i ~(1 \leq i \leq 2)$ using Lemma~\ref{all_above_below}.
\begin{equation*}
u_{1}^{a} = \alpha_{1} - |b_1| = 6 - 1 = 5, \quad u_{2}^{b} =  \alpha_{2} - |b_2| = 2 - 1 = 1,
\end{equation*}
\begin{equation*}
u_{3}^{a} = \alpha_{3} - |b_2| = 4 - 1 = 3, \quad u_{4}^{b} =  \alpha_{4} - |b_2| = 2 - 1 = 1,
\end{equation*}
\begin{equation*}
u_{5}^{a} = \alpha_{5} - |b_3| = 5 - 1 = 4, \quad u_{6}^{b} =  \alpha_{6} - |b_3| = 1 - 1 = 0.
\end{equation*}
\noindent Therefore, we have $5$ above components and $1$ below component in $U_1$, $3$ above components and $1$ below component in $U_2$ and $4$ above components and no below component in $U_3$. 

Since $|T_1| \neq 0$ and $\beta_{4} <  \beta_{5}$ in $G_1$,
\begin{eqnarray*}
	u_{1}^{va} & =  & \xi_{1} - |T_1| - \max\{n_1 - d_{2}, T_1\} + \max\{0, T_1\} - l_1 \\
	& = & 5 - 1 - \max\{2 - 2, 1\} + \max\{0, 1\} - 0 \\
	& = & 4
	\end{eqnarray*}
\noindent and 
\begin{eqnarray*}
u_{2}^{vb} & =  & \xi_{2} - |T_1| - \max\{n_1 - d_{1}, -T_1\} + \max\{0, -T_1\} - l_1 \\
            & =  & 4 - 1 - \max\{2 - 0, -1\} + \max\{0, -1\} - 0 \\
            & = & 1.
\end{eqnarray*}
\noindent Also,
\begin{equation*}
u_{1}^{v'a} = \xi^{'}_{1} - l^{'}_1 = 4 - 1 = 3 \quad \mbox{ and } \quad u_{2}^{v'b} = \xi^{'}_{2} - l^{'}_1 = 1 - 1 = 0.
\end{equation*}
\noindent There are  $4$ visible above components, $1$ visible below component, $3$ invisible above components and no invisible below component in $G_1$. 

Since $|T_2| \neq 0$ and $\beta_{5} >  \beta_{6}$ in $G_2$,
\begin{eqnarray*}
	u_{3}^{va} & =  & \xi_{3} - |T_2| - \max\{c_2 - k_2 - d_{4}, T_2\} + \max\{0, T_2\} - l_2 \\
	& = & 6 - 4 - \max\{3 - 3 - 0, -4\} + \max\{0, -4\} - 1 \\
	& = & 1
\end{eqnarray*}
\noindent and 
\begin{eqnarray*}
	u_{4}^{vb} & =  & \xi_{4} - |T_2| - \max\{c_2 - k_2 - d_{3}, -T_2\} + \max\{0, -T_2\} - l_2 \\
	& =  & 6 - 4 - \max\{3 - 3 - 0, 4\} + \max\{0, 4\} - 1 \\
	& = & 1.
\end{eqnarray*}
\noindent Also,
\begin{equation*}
u_{3}^{v'a} = \xi^{'}_{3} - l^{'}_2 = 0 - 0 = 0 \quad \mbox{ and } \quad u_{4}^{v'b} = \xi^{'}_{4} - l^{'}_2 = 0 - 0 = 0.
\end{equation*}
\noindent There are  $1$ visible above component, $1$ visible below component, no invisible above component and no invisible below component in $G_2$. The calculated path components in each $U_i$, $G_i$ and $G^{*}$ are connected in a unique way up to isotopy and thus,  the multiple curve $L$ in Figure~\ref{gen_example} is determined uniquely.
\begin{figure}[!ht]
	\centering
	\includegraphics[width=0.83\textwidth]{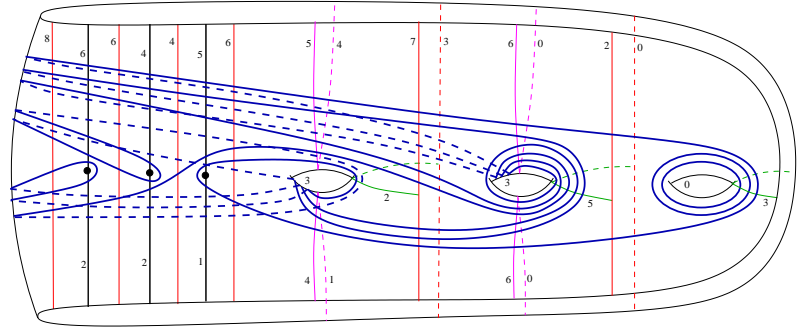}
	\caption{ The multiple curve $L$ with the intersection numbers $(6, 2, 4, 2, 5, 1; 8, 6, 4, 6, 7, 2; 3, 0; 5, 4, 6, 6; 4, 1, 0, 0; 2, 5, 3; 3, 3; 0)$ }\label{gen_example}
\end{figure} 

\end{example}

\bibliographystyle{amsplain}
\providecommand{\bysame}{\leavevmode\hbox
to3em{\hrulefill}\thinspace}
\providecommand{\MR}{\relax\ifhmode\unskip\space\fi MR }
\providecommand{\MRhref}[2]{%
  \href{http://www.ams.org/mathscinet-getitem?mr=#1}{#2}
} \providecommand{\href}[2]{#2}

\end{document}